\documentclass[10pt]{amsart}
\usepackage{amssymb}
\usepackage{bm}
\usepackage{graphicx}
\usepackage[centertags]{amsmath}
\usepackage{amsfonts}
\usepackage{amsthm}
\newtheorem{thm}{Theorem}
\newtheorem{cor}[thm]{Corollary}
\newtheorem{lem}[thm]{Lemma}

\newtheorem{prop}[thm]{Proposition}
\newtheorem{claim}[thm]{Claim}
\newtheorem{fact}[thm]{Fact}
\newtheorem{problem}{Problem}
\newtheorem{defn}[thm]{Definition}
\theoremstyle{definition}

\newcommand{\rr}{\mathbb{R}}
\newcommand{\nn}{\mathbb{N}}
\newcommand{\ee}{\varepsilon}

\newcommand{\sbs}{\mathrm{SB}}

\newcommand{\ncn}{\mathrm{NC}}
\newcommand{\nell}{\mathrm{NC}_{\ell_1}}

\newcommand{\aaa}{\mathcal{A}}
\newcommand{\bbb}{\mathcal{B}}
\newcommand{\ccc}{\mathcal{C}}
\newcommand{\ddd}{\mathcal{D}}
\newcommand{\eee}{\mathcal{E}}
\newcommand{\sss}{\mathcal{S}}
\newcommand{\xxx}{\mathcal{X}}
\newcommand{\tr}{\mathrm{Tr}}
\newcommand{\wf}{\mathrm{WF}}
\newcommand{\chains}{\mathrm{chains}}
\newcommand{\ct}{2^{<\nn}}
\newcommand{\ltr}{\Lambda^{<\nn}}
\newcommand{\sg}{\sigma}
\newcommand{\con}{\smallfrown}
\newcommand{\supp}{\mathrm{supp}}
\newcommand{\range}{\mathrm{range}}

\begin{document}

\title{Quotients of Banach spaces and surjectively universal spaces}
\author{Pandelis Dodos}

\address{Department of Mathematics, University of Athens, Panepistimiopolis 157 84, Athens, Greece.}
\email{pdodos@math.ntua.gr}

\thanks{Research supported by NSF grant DMS-0903558.}

\maketitle


\begin{abstract}
We characterize those classes $\ccc$ of separable Banach spaces for
which there exists a separable Banach space $Y$ not containing $\ell_1$
and such that every space in the class $\ccc$ is a quotient of $Y$.
\end{abstract}


\section{Introduction}

There are two classical universality results in Banach Space Theory.
The first one, known to Stefan Banach \cite{Ba}, asserts that the space
$C(2^\nn)$, where $2^\nn$ stands for the Cantor set, is isometrically
universal for all separable Banach spaces; that is, every separable
Banach space is isometric to a subspace of $C(2^\nn)$. The second result,
also known to Banach, is ``dual" to the previous one and asserts that
every separable Banach space is isometric to a quotient of $\ell_1$.

By now, it is well understood that there are natural classes of separable
Banach spaces for which one cannot get something better from what it is
quoted above (see \cite{Ar,Bou,GK,Sz}). For instance, if a separable Banach
space $Y$ is universal for the separable reflexive Banach spaces, then $Y$
must contain an isomorphic copy of $C(2^\nn)$, and so, it is universal for
all separable Banach spaces. However, there are non-trivial classes of
separable Banach spaces which do admit ``smaller" universal spaces (see
\cite{AD,D2,DF,DL,FOSZ,OS,OSZ,Pr}).

Recently, in \cite{D2}, a characterization was obtained of those classes of separable
Banach spaces admitting a universal space which is not universal for all
separable Banach spaces. One of the goals of the present paper is to obtain
the corresponding characterization for the ``dual" problem concerning
quotients instead of embeddings. To proceed with our discussion it is
useful to introduce the following definition.
\begin{defn} \label{d1}
We say that a Banach space $Y$ is a \emph{surjectively universal space}
for a class $\ccc$ of Banach spaces if every space in the class $\ccc$
is a quotient\footnote[1]{If $X$ and $Y$ are Banach spaces, then we say
that $X$ is a \textit{quotient} of $Y$ if there exists a bounded, linear
and onto operator $Q:Y\to X$.} of $Y$.
\end{defn}
We can now state the main problem addressed in this paper.
\begin{enumerate}
\item[\textbf{(P)}] Let $\ccc$ be a class of separable Banach space. When
can we find a separable Banach space $Y$ which is surjectively universal
for the class $\ccc$ and does not contain a copy of $\ell_1$?
\end{enumerate}
We notice that if a separable Banach space $Y$ does not contain a copy
of $\ell_1$, then $\ell_1$ is not a quotient of $Y$ (see
\cite[Proposition 2.f.7]{LT}) and therefore $Y$ is not surjectively
universal for all separable Banach spaces.

To state our results we recall the following (more or less standard)
notation and terminology. By $\sbs$ we denote the standard Borel space
of separable Banach spaces defined by B. Bossard \cite{Bos2}, by $\nell$
we denote the subset of $\sbs$ consisting of all $X\in\sbs$ not containing
an isomorphic copy of $\ell_1$ and, finally, by $\phi_{\nell}$ we denote
Bourgain's $\ell_1$ index \cite{Bou} (these concepts are properly defined
in \S 2). We show the following.
\begin{thm} \label{t2}
Let $\ccc\subseteq \sbs$. Then the following are equivalent.
\begin{enumerate}
\item[(i)] There exists a separable Banach space $Y$ which is surjectively
universal for the class $\ccc$ and does not contain a copy of $\ell_1$.
\item[(ii)] We have $\sup\{\phi_{\nell}(X):X\in\ccc\}<\omega_1$.
\item[(iii)] There exists an analytic subset $\aaa$ of $\nell$ with
$\ccc\subseteq \aaa$.
\end{enumerate}
\end{thm}
We notice that stronger versions of Theorem \ref{t2} are valid
provided that all spaces in the class $\ccc$ have some additional property
(see \S 5).

A basic ingredient of the proof of Theorem \ref{t2} (an ingredient which is
probably of independent interest) is the construction for every separable Banach
space $X$ of a Banach space $E_X$ with special properties.
Specifically we show the following.
\begin{thm} \label{t3}
Let $X$ be a separable Banach space. Then there exists a separable Banach
space $E_X$ such that the following are satisfied.
\begin{enumerate}
\item[(i)] \emph{(Existence of a Schauder basis)} The space $E_X$ has
a normalized monotone Schauder basis $(e_n^X)$.
\item[(ii)] \emph{(Existence of a quotient map)} There exists a
norm-one linear and onto operator $Q_X:E_X\to X$.
\item[(iii)] \emph{(Subspace structure)} If $Y$ is an infinite-dimensional
subspace of $E_X$ and the operator $Q_X:Y\to X$ is strictly singular,
then $Y$ contains a copy of $c_0$.
\item[(iv)] \emph{(Representability of $X$)} For every normalized
basic sequence $(w_k)$ in $X$ there exists a subsequence
$(e_{n_k}^X)$ of $(e_n^X)$ such that $(e^X_{n_k})$ is equivalent
to $(w_k)$.
\item[(v)] \emph{(Uniformity)} The set $\eee\subseteq \sbs\times\sbs$
defined by
\[ (X,Y)\in\eee\Leftrightarrow Y \text{ is isometric to } E_X\]
is analytic.
\item[(vi)] \emph{(Preservation of separability of the dual)} $E^*_X$ is
separable if and only if $X^*$ is separable.
\end{enumerate}
\end{thm}
We notice that there exists a large number of related results found
in the literature; see, for instance, \cite{DFJP,FOS,FOSZ,JZ,OS,OSZ,Z}.
The novelty in Theorem \ref{t3} is that, beside functional analytic tools, its
proof is enriched with descriptive set theory and the combinatorial
machinery developed in \cite{ADK1} and \cite{ADK2}.

The paper is organized as follows. In \S 2 we gather some background
material. In \S 3 we define the space $E_X$ and we give the proof
of Theorem \ref{t3}. The proof of Theorem \ref{t2} (actually of a more
detailed version of it) is given in \S 4. Finally, in \S 5 we present
some related results and we discuss open problems.


\section{Background material}

Our general notation and terminology is standard as can be found, for instance,
in \cite{LT} and \cite{Kechris}. By $\nn=\{0,1,2,...\}$ we shall denote the
natural numbers.

We will frequently need to compute the descriptive set-theoretic complexity
of various sets and relations. To this end, we will use the ``Kuratowski-Tarski
algorithm". We will assume that the reader is familiar with this classical method.
For more details we refer to \cite[page 353]{Kechris}.

\subsection{Trees}

Let $\Lambda$ be a non-empty set. By $\ltr$ we denote the set of all finite
sequences in $\Lambda$ while by $\Lambda^\nn$ we denote the set of all
infinite sequences in $\Lambda$ (the empty sequence is denoted by $\varnothing$
and is included in $\ltr$). We view $\ltr$ as a tree equipped with the (strict) partial
order $\sqsubset$ of extension. Two nodes $s,t\in\ltr$ are said to be
\textit{comparable} if either $s\sqsubseteq t$ or $t\sqsubseteq s$. Otherwise,
$s$ and $t$ are said to be \textit{incomparable}. A subset of $\ltr$ consisting
of pairwise comparable nodes is said to be a \textit{chain}, while a subset of
$\ltr$ consisting of pairwise incomparable nodes is said to be an \textit{antichain}.

A \textit{tree} $T$ on $\Lambda$ is a subset of $\ltr$ which is closed under
initial segments. By $\tr(\Lambda)$ we denote the set of all trees on
$\Lambda$. Hence
\[ T\in\tr(\Lambda) \Leftrightarrow \forall s,t\in\ltr \ (s\sqsubseteq t
\text{ and } t\in T\Rightarrow s\in T). \]
The \textit{body} of a tree $T$ on $\Lambda$ is defined to be the set
$\{\sigma\in\Lambda^\nn: \sigma|n\in T \ \forall n\in\nn\}$ and is denoted
by $[T]$. A tree $T$ is said to be \textit{well-founded} if $[T]=\varnothing$.
By $\wf(\Lambda)$ we denote the set of all well-founded trees on $\Lambda$.
For every $T\in\wf(\Lambda)$ we let
$T'=\{s\in T: \exists t\in T \text{ with } s\sqsubset t\}\in\wf(\Lambda)$.
By transfinite recursion, we define the iterated derivatives $T^{\xi}$
$(\xi<\kappa^+)$ of $T$, where $\kappa$ stands for the cardinality of $\Lambda$.
The \textit{order} $o(T)$ of $T$ is defined to be the least ordinal $\xi$
such that $T^{\xi}=\varnothing$.

Let $S$ and $T$ be trees on two non-empty sets $\Lambda_1$ and $\Lambda_2$
respectively. A map $\psi:S\to T$ is said to be \textit{monotone} if for
every $s_0, s_1\in S$ with $s_0\sqsubset s_1$ we have $\psi(s_0)\sqsubset
\psi(s_1)$. We notice that if there exists a monotone map $\psi:S\to T$
and $T$ is well-founded, then $S$ is well-founded and $o(S)\leq o(T)$.

\subsection{Dyadic subtrees and related combinatorics}

Let $\ct$ be the Cantor tree; i.e. $\ct$ is the set of all finite
sequences of $0$'s and $1$'s. For every $s,t\in\ct$ we let $s\wedge t$
be the $\sqsubset$-maximal node $w$ of $\ct$ with $w\sqsubseteq s$ and
$w\sqsubseteq t$. If $s,t\in\ct$ are incomparable with respect to $\sqsubseteq$,
then we write $s\prec t$ provided that $(s\wedge t)^{\con}0 \sqsubseteq s$
and $(s\wedge t)^{\con}1\sqsubseteq t$. We say that a subset $D$ of $\ct$
is a \textit{dyadic subtree} of $\ct$ if $D$ can be written in the form
$\{d_t:t\in\ct\}$ so that for every $t_0, t_1\in\ct$ we have $t_0\sqsubset t_1$
(respectively $t_0\prec t_1$) if and only if $d_{t_0} \sqsubset d_{t_1}$
(respectively $d_{t_0}\prec d_{t_1}$). It is easy to see that such a
representation of $D$ as $\{d_t: t\in\ct\}$ is unique. In the sequel
when we write $D=\{d_t:t\in\ct\}$, where $D$ is a dyadic subtree, we will
assume that this is the canonical representation of $D$ described above.

For every dyadic subtree $D$ of $\ct$ by $[D]_{\chains}$ we denote the set
of all infinite chains of $D$. Notice that $[D]_{\chains}$ is a $G_\delta$,
hence Polish, subspace of $2^{\ct}$. We will need the following partition
theorem due to J. Stern (see \cite{Ste}).
\begin{thm} \label{t4}
Let $D$ be a dyadic subtree of $\ct$ and $\xxx$ be an analytic subset of
$[D]_{\chains}$. Then there exists a dyadic subtree $S$ of $\ct$ with
$S\subseteq D$ and such that either $[S]_{\chains}\subseteq \xxx$
or $[S]_{\chains}\cap \xxx=\varnothing$.
\end{thm}

\subsection{Separable Banach spaces with non-separable dual}

We will need a structural result concerning separable Banach spaces with
non-separable dual. To state this result and to facilitate future references
to it, it is convenient to introduce the following definition.
\begin{defn}
\label{d5} Let $X$ be a Banach space and $(x_t)_{t\in\ct}$ be a sequence
in $X$ indexed by the Cantor tree. We say that $(x_t)_{t\in\ct}$ is
\emph{topologically equivalent to the basis of James tree} if the following
are satisfied.
\begin{enumerate}
\item[(1)] The sequence $(x_t)_{t\in\ct}$ is semi-normalized.
\item[(2)] For every infinite antichain $A$ of $\ct$ the sequence
$(x_t)_{t\in A}$ is weakly null.
\item[(3)] For every $\sg\in 2^\nn$ the sequence $(x_{\sg|n})$ is weak* convergent
to an element $x^{**}_\sg\in X^{**}\setminus X$. Moreover, if $\sg, \tau\in 2^\nn$
with $\sg\neq \tau$, then $x^{**}_\sg\neq x^{**}_\tau$.
\end{enumerate}
\end{defn}
The archetypical example of such a sequence is the standard Schauder basis of
the space $JT$ (see \cite{Ja}). There are also classical Banach spaces having
a natural Schauder basis which is topologically equivalent to the basis of
James tree; the space $C(2^\nn)$ is an example. We isolate, for future use,
the following fact.
\begin{fact} \label{f6}
Let $X$ be a Banach space and $(x_t)_{t\in\ct}$ be a sequence in $X$ which is
topologically equivalent to the basis of James tree. Then for every dyadic
subtree $D=\{d_t:t\in\ct\}$ of $\ct$ the sequence $(x_{d_t})_{t\in\ct}$ is
topologically equivalent to the basis of James tree.
\end{fact}
We notice that if a Banach space $X$ contains a sequence $(x_t)_{t\in\ct}$
which is topologically equivalent to the basis of James tree, then $X^*$
is not separable. The following theorem establishes the converse for
separable Banach spaces not containing a copy of $\ell_1$ (see
\cite[Theorem 40]{ADK1} or \cite[Theorem 17]{ADK2}).
\begin{thm} \label{t7}
Let $X$ be a separable Banach space not containing a copy of $\ell_1$ and with
non-separable dual. Then $X$ contains a sequence $(x_t)_{t\in\ct}$ which is
topologically equivalent to the basis of James tree.
\end{thm}

\subsection{Co-analytic ranks}

Let $(X, \Sigma)$ be a standard Borel space; that is, $X$ is a set, $\Sigma$
is a $\sigma$-algebra on $X$ and the measurable space $(X, \Sigma)$ is Borel
isomorphic to the reals. A subset $A$ of $X$ is said to be \textit{analytic}
if there exists a Borel map $f:\nn^\nn\to X$ with $f(\nn^\nn)=A$. A subset
of $X$ is said to be \textit{co-analytic} if its complement is analytic.
Now let $B$ be a co-analytic subset of $X$. A map $\phi:B\to\omega_1$ is said
to be a \textit{co-analytic rank on} $B$ if there exist relations
$\leq_\Sigma$ and $\leq_\Pi$ in $X\times X$ which are analytic and co-analytic
respectively and are such that for every $y\in B$ we have
\[ x\in B \text{ and } \phi(x)\leq \phi(y) \Leftrightarrow x\leq_\Sigma y
\Leftrightarrow x\leq_\Pi y. \]
For our purposes, the most important property of co-analytic ranks is that
they satisfy \textit{boundedness}. This means that if $\phi:B\to\omega_1$
is a co-analytic rank on a co-analytic set $B$ and $A\subseteq B$ is analytic,
then $\sup\{\phi(x):x\in A\}<\omega_1$. For a proof as well as for a thorough
presentation of Rank Theory we refer to \cite[\S 34]{Kechris}.

\subsection{The standard Borel space of separable Banach spaces}

Let $F\big(C(2^\nn)\big)$ be the set of all closed subsets of $C(2^\nn)$ and
$\Sigma$ be the Effors-Borel structure on $F\big(C(2^\nn)\big)$; that is, $\Sigma$
is the $\sigma$-algebra generated by the sets
\[ \big\{ F\in F\big(C(2^\nn)\big): F\cap U\neq\varnothing\big\}\]
where $U$ ranges over all open subsets of $C(2^\nn)$. Consider the set
\[ \sbs=\big\{ X\in F\big(C(2^\nn)\big): X \text{ is a linear subspace}\big\}.\]
It is easy to see that the set $\sbs$ equipped with the relative Effors-Borel
structure is a standard Borel space (see \cite{Bos2} for more details). The
space $\sbs$ is referred in the literature as the \textit{standard Borel
space of separable Banach spaces}. We will need the following consequence
of the Kuratowski--Ryll-Nardzewski selection Theorem (see \cite[Theorem 12.13]{Kechris}).
\begin{prop} \label{p8}
There exists a sequence $S_n:\sbs\to C(2^\nn)$ $(n\in\nn)$ of Borel maps
such that for every $X\in\sbs$ with $X\neq \{0\}$ we have $S_n(X)\in S_X$ and
the sequence $(S_n(X))$ is norm dense in $S_X$, where $S_X$ stands for
the unit sphere of $X$.
\end{prop}

\subsection{The class $\ncn_Z$ and Bourgain's indices}

Let $Z$ be a Banach space with a Schauder basis\footnote[2]{Throughout
the paper when we say that a Banach space $X$ has a Schauder basis or the
bounded approximation property, then we implicitly assume that $X$ is
infinite-dimensional.}. We fix a normalized Schauder basis $(z_n)$ of $Z$.
If $Z$ is one of the classical sequence spaces $c_0$ and $\ell_p$
$(1\leq p<+\infty)$, then we let $(z_n)$ be the standard unit vector
basis. We consider the set
\[ \ncn_Z=\{ X\in\sbs: X \text{ does not contain an isomorphic copy of } Z\}.\]
Let $\delta\geq 1$ and let $Y$ be an arbitrary separable Banach space. Following
J. Bourgain \cite{Bou}, we introduce a tree $\mathbf{T}(Y,Z,(z_n),\delta)$
on $Y$ defined by the rule
\[ (y_n)_{n=0}^k\in\mathbf{T}(Y,Z,(z_n),\delta) \ \Leftrightarrow \
(y_n)_{n=0}^k \text{ is } \delta-\text{equivalent to } (z_n)_{n=0}^k.\]
In particular, if $Z$ is the space $\ell_1$, then for every $\delta\geq 1$
and every finite sequence $(y_n)_{n=0}^k$ in $Y$ we have
$(y_n)_{n=0}^k\in \mathbf{T}(Y,\ell_1,(z_n),\delta)$ if and only if
for every $a_0,..., a_k\in\rr$ it holds that
\[ \frac{1}{\delta} \sum_{n=0}^k |a_n| \leq \big\| \sum_{n=0}^k a_n y_n
\big\| \leq \delta \sum_{n=0}^k |a_n|. \]
We notice that $Y\in\ncn_Z$ if and only if for every $\delta\geq 1$ the tree
$\mathbf{T}(Y,Z,(z_n),\delta)$ is well-founded. We set $\phi_{\ncn_Z}(Y)=\omega_1$
if $Y\notin\ncn_Z$, while if $Y\in\ncn_Z$ we define
\begin{equation}
\label{e1} \phi_{\ncn_Z}(Y)=\sup\big\{ o\big(\mathbf{T}(Y,Z,(z_n),\delta)\big):
\delta\geq 1 \big\}.
\end{equation}
In \cite{Bou}, Bourgain proved that for every Banach space $Z$ with a
Schauder basis and every $Y\in\sbs$ we have $Y\in\ncn_Z$ if and only if
$\phi_{\ncn_Z}(Y)<\omega_1$. We need the following refinement of this result
(see \cite[Theorem 4.4]{Bos2}).
\begin{thm} \label{t9}
Let $Z$ be a Banach space with a Schauder basis. Then the set $\ncn_Z$
is co-analytic and the map $\phi_{\ncn_Z}:\ncn_Z\to\omega_1$ is a
co-analytic rank on $\ncn_Z$.
\end{thm}
We will also need the following quantitative strengthening of the
classical fact that $\ell_1$ has the lifting property.
\begin{lem} \label{l10}
Let $X$ and $Y$ be separable Banach spaces and assume that $X$ is a quotient
of $Y$. Then $\phi_{\ncn_{\ell_1}}(X)\leq \phi_{\ncn_{\ell_1}}(Y)$.
\end{lem}
\begin{proof}
Clearly we may assume that $Y$ does not contain a copy of $\ell_1$.
We fix a quotient map $Q:Y\to X$. There exists a constant $C\geq 1$ such that
\begin{enumerate}
\item[(a)] $\|Q\|\leq C$ and
\item[(b)] for every $x\in X$ there exists $y\in Y$ with $Q(y)=x$ and
$\|y\|\leq C \|x\|$.
\end{enumerate}
For every $x\in X$ we select $y_x\in Y$ such that $Q(y_x)=x$ and
$\|y_x\|\leq C\|x\|$. We define a map $\psi:X^{<\nn}\to Y^{<\nn}$
as follows. We set $\psi(\varnothing)=\varnothing$.
If $s=(x_n)_{n=0}^k\in X^{<\nn}\setminus \{\varnothing\}$, then we set
\[ \psi(s)=(y_{x_n})_{n=0}^k. \]
We notice that the map $\psi$ is monotone. Denote by $(z_n)$ the standard unit
vector basis of $\ell_1$.
\begin{claim} \label{c11}
For every $\delta\geq 1$ if $s\in\mathbf{T}(X,\ell_1, (z_n),\delta)$,
then $\psi(s)\in \mathbf{T}(Y,\ell_1, (z_n),C\delta)$.
\end{claim}
Granting Claim \ref{c11}, the proof of the lemma is completed. Indeed,
by Claim \ref{c11}, we have that for every $\delta\geq 1$ the map $\psi$
is a monotone map from the tree $\mathbf{T}(X,\ell_1, (z_n),\delta)$ into
the tree $\mathbf{T}(Y,\ell_1, (z_n),C\delta)$. Hence
\[ o\big(\mathbf{T}(X,\ell_1, (z_n),\delta)\big) \leq
o\big(\mathbf{T}(Y,\ell_1, (z_n),C\delta)\big). \]
The above estimate clearly implies that
$\phi_{\ncn_{\ell_1}}(X)\leq \phi_{\ncn_{\ell_1}}(Y)$.

It remains to prove Claim \ref{c11}. To this end let $s=(x_n)_{n=0}^k
\in \mathbf{T}(X,\ell_1, (z_n),\delta)$. Also let $a_0,..., a_k\in\rr$
be arbitrary. Notice that
\[ Q(a_0 y_{x_0}+... +a_k y_{x_k})= a_0 x_0+ ... + a_k x_k.\]
Hence, by (a), we get
\begin{equation} \label{e2}
\frac{1}{\delta} \sum_{n=0}^k |a_n| \leq \big\| \sum_{n=0}^k a_n x_n\big\|=
\big\| Q\Big( \sum_{n=0}^k a_n y_{x_n}\Big) \big\| \leq
C \big\| \sum_{n=0}^k a_n y_{x_n} \big\|.
\end{equation}
Observe that $\|x_n\|\leq\delta$ for every $n\in\{0,...,k\}$. Therefore,
\begin{equation} \label{e3}
\big\| \sum_{n=0}^k a_n y_{x_n} \big\| \leq
\sum_{n=0}^k |a_n| \cdot \|y_{x_n}\|=
\sum_{n=0}^k |a_n| \cdot \|Q(x_n)\| \leq C\delta \sum_{n=0}^k |a_n|.
\end{equation}
Since the coefficients $a_0, ..., a_k\in\rr$ were arbitrary, inequalities
(\ref{e2}) and (\ref{e3}) imply that $\psi(s)=(y_{x_n})_{n=0}^k\in \mathbf{T}(Y,\ell_1,
(z_n),C\delta)$. This completes the proof of Claim \ref{c11}, and as we have
indicated above, the proof of the lemma is also completed.
\end{proof}

\subsection{Separable spaces with the B. A. P. and Lusky's Theorem}

By the results in \cite{JRZ} and \cite{Pe2}, a separable Banach space $X$
has the bounded approximation property (in short B. A. P.) if and only if
$X$ is isomorphic to a complemented subspace of a Banach space $Y$ with a
Schauder basis. W. Lusky found an effective way to produce the space $Y$.
To state his result we need, first, to recall the definition of the space
$C_0$ due to W. B. Johnson. Let $(F_n)$ be a sequence of finite-dimensional
spaces dense in the Banach-Mazur distance in the class of all finite-dimensional
spaces. We set
\begin{equation} \label{e4}
C_0=\Big( \sum_{n\in\nn} \oplus F_n \Big)_{c_0}
\end{equation}
and we notice that $C_0$ is hereditarily $c_0$ (i.e. every infinite-dimensional
subspace of $C_0$ contains a copy of $c_0$). We can now state Lusky's Theorem
(see \cite{Lu}).
\begin{thm} \label{t12}
Let $X$ be a separable Banach space with the bounded approximation property.
Then $X\oplus C_0$ has a Schauder basis.
\end{thm}
Theorem \ref{t12} will be used in the following parameterized form.
\begin{lem} \label{l13}
Let $Z$ be a minimal\footnote[3]{We recall that an infinite-dimensional
Banach space $Z$ is said to be minimal if every infinite-dimensional
subspace of $Z$ contains an isomorphic copy of $Z$; e.g. the classical
sequence spaces $c_0$ and $\ell_p$ $(1\leq p<+\infty)$ are minimal spaces.}
Banach space not containing a copy of $c_0$. Let $\aaa$ be an analytic subset
of $\ncn_Z\cap\ncn_{\ell_1}$. Then there exists a (possibly empty) subset
$\ddd$ of $\ncn_Z\cap\ncn_{\ell_1}$ with the following properties.
\begin{enumerate}
\item[(i)] The set $\ddd$ is analytic.
\item[(ii)] Every $Y\in\ddd$ has a Schauder basis.
\item[(iii)] For every $X\in\aaa$ with the bounded approximation property there
exists $Y\in\ddd$ such that $X$ is isomorphic to a complemented subspace of $Y$.
\end{enumerate}
\end{lem}
\begin{proof}
The result is essentially known, and so, we will be rather sketchy. First
we consider the set $\bbb\subseteq\sbs$ defined by
\[ X\in\bbb\Leftrightarrow X \text{ has the bounded approximation property.} \]
Using the characterization of B. A. P. mentioned above, it is easy to check
that the set $\bbb$ is analytic. Next, consider the set
$\ccc\subseteq \sbs\times\sbs$ defined by
\[ (X,Y)\in\ccc \Leftrightarrow Y \text{ is isomorphic to } X\oplus C_0.\]
It is also easy to see that $\ccc$ is analytic (see \cite{AD} for more details).
Define $\ddd\subseteq \sbs$ by the rule
\[ Y\in\ddd\Leftrightarrow \exists X \ [X\in \aaa\cap\bbb \text{ and }
(X,Y)\in\ccc]\]
and notice that $\ddd$ is analytic. By Theorem \ref{t12}, the set
$\ddd$ is as desired.
\end{proof}

\subsection{Amalgamated spaces}

A recurrent theme in the proof of various universality results found
in the literature (a theme that goes back to the classical results of
A. Pe{\l}czy\'{n}ski \cite{Pe1}) is the use at a certain point of a
``gluing" procedure. A number of different ``gluing" procedures have
been proposed by several authors. We will need the following result
(see \cite[Theorem 71]{AD}).
\begin{thm} \label{t14}
Let $1<p<+\infty$ and $\ccc$ be an analytic subset of $\sbs$ such that
every $Y\in\ccc$ has a Schauder basis. Then there exists a Banach space
$V$ with a Schauder basis that contains a complemented copy of every
space in the class $\ccc$.

Moreover, if $W$ is an infinite-dimensional subspace of $V$, then either
\begin{enumerate}
\item[(i)] $W$ contains a copy of $\ell_p$, or
\item[(ii)] there exists $Y_0,..., Y_n$ in the class $\ccc$ such that $W$
is isomorphic to a subspace of $Y_0\oplus ... \oplus Y_n$.
\end{enumerate}
\end{thm}
The space $V$ obtained above is called in \cite{AD} as the
$p$-\textit{amalgamation space} of the class $\ccc$. The reader can find
in \cite[\S 8]{AD} an extensive study of its properties.


\section{Quotients of Banach spaces}

\subsection{Definitions}

We start with the following.
\begin{defn} \label{d15}
Let $X$ be a separable Banach space and $(x_n)$ be a sequence (with possible
repetitions) in the unit sphere of $X$ which is norm dense in $S_X$. By $E_X$
we shall denote the completion of $c_{00}(\nn)$ under the norm
\begin{equation} \label{e5}
\|z\|_{E_X}= \sup\Big\{ \big\|\sum_{n=0}^m z(n) x_n\big\|_X: m\in\nn\Big\}.
\end{equation}
By $(e_n^X)$ we shall denote the standard Hamel basis of $c_{00}(\nn)$
regarded as a sequence in $E_X$. If $X=\{0\}$, then by convention we set $E_X=c_0$.
\end{defn}
The construction of the space $E_X$ is somehow ``classical" and the motivation
for the above definition can be traced in the proof of the fact that every
separable Banach space is a quotient of $\ell_1$  (see \cite[page 108]{LT}).
A similar construction was presented by G. Schechtman in \cite{Sch}
for different, though related, purposes.

We isolate two elementary properties of the space $E_X$. First, we observe
that the sequence $(e_n^X)$ defines a normalized monotone Schauder basis of $E_X$.
It is also easy to see that the map $E_X\ni e_n^X\mapsto x_n\in X$ is extended
to a norm-one linear operator. This operator will be denoted as follows.
\begin{defn} \label{d16}
By $Q_X:E_X\to X$ we shall denote the (unique) bounded linear operator
satisfying $Q_X(e^X_n)=x_n$ for every $n\in\nn$.
\end{defn}
Let us make at this point two comments about the above definitions.
Let $(y_n)$ be a basic sequence in a Banach space $Y$ and assume
that the map
\[ \overline{\mathrm{span}}\{y_n:n\in\nn\}\ni y_n\mapsto x_n\in X\]
is extended to a bounded linear operator. Then it is easy to see that
there exists a constant $C\geq 1$ such that the sequence $(e_n^X)$
is $C$-dominated\footnote[4]{We recall that if $(v_n)$ and $(y_n)$ are two
basic sequences in two Banach spaces $V$ and $Y$ respectively, then $(v_n)$ is
said to be $C$-dominated by $(y_n)$ if for every $k\in\nn$ and every
$a_0,...,a_k\in\rr$ we have $\| \sum_{n=0}^k a_n v_n \|_V \leq
C \|\sum_{n=0}^k a_n y_n \|_Y$.} by the sequence $(y_n)$. In other words,
among all basic sequences $(y_n)$ with the property that the
map $\overline{\mathrm{span}}\{y_n:n\in\nn\}\ni y_n\mapsto x_n\in X$ is
extended to a bounded linear operator, the sequence $(e_n^X)$ is the
\textit{minimal} one with respect to domination.

Notice also that the space $E_X$ depends on the choice of the sequence
$(x_n)$. For our purposes, however, the dependence is not important as can
be seen from the simple following observation. Let $(d_n)$ be another
sequence in the unit sphere of $X$ which is norm dense in $S_X$ and
let $E'_X$ be the completion of $c_{00}(\nn)$ under the norm
\[ \|z\|_{E'_X}=\sup\Big\{ \big\|\sum_{n=0}^m z(n) d_n\|_X:
m\in\nn\Big\}.\]
Then it is easy to check that $E_X$ embeds isomorphically into
$E'_X$ and vice versa. Actually, it is possible to modify the
construction to obtain a different space sharing most of the properties
of the space $E_X$ and not depending on the choice of the dense sequence.
We could not find, however, any application of this construction
and since it is involved and conceptually less natural to grasp
we prefer not to bother the reader with it.

The rest of the section is organized as follows. In \S 3.2 we present
some preliminary tools needed for the proof of Theorem \ref{t3}.
The proof of Theorem \ref{t3} is given in \S 3.3 while in \S 3.4
we present some its consequences. Finally, in \S 3.5 we make some
comments.

\subsection{Preliminary tools}

We start by introducing some pieces of notation that will be
used only in this section. Let $F$ and $G$ be two non-empty
finite subsets of $\nn$. We write $F<G$ if $\max F<\min G$.
Let $(e_n)$ be a basic sequence in a Banach space $E$ and
let $v$ be a vector in $\overline{\mathrm{span}}\{e_n:n\in\nn\}$.
There exists a (unique) sequence $(a_n)$ of reals such that
$v=\sum_{n\in\nn} a_n e_n$. The \textit{support} of the vector
$v$, denoted by $\supp(v)$, is defined to be the set
$\{n\in\nn: a_n\neq 0\}$. The \textit{range} of the vector $v$,
denoted by $\range(v)$, is defined to be the minimal interval
of $\nn$ that contains $\supp(v)$.

In what follows $X$ will be a separable Banach space
and $(x_n)$ will be the sequence in $X$ which is used to define the
space $E_X$. The following propositions will be basic tools for the
analysis of the space $E_X$.
\begin{prop} \label{p17}
Let $(v_k)$ be a semi-normalized block sequence of $(e^X_n)$ and assume
that $\|Q_X(v_k)\|_X\leq 2^{-k}$ for every $k\in\nn$. Then the sequence
$(v_k)$ is equivalent to the standard unit vector basis of $c_0$.
\end{prop}
\begin{proof}
We select a constant $C>0$ such that $\|v_k\|_{E_X}\leq C$ for every $k\in\nn$.
Let $F=\{k_0< ...< k_j\}$ be a finite subset of $\nn$. We will show that
\[ \big\| \sum_{i=0}^j v_{k_i} \big\|_{E_X} \leq 2+C.\]
This will finish the proof. To this end we argue as follows. First we set
\begin{enumerate}
\item[(a)] $G_i=\supp(v_{k_i})$ and $m_i=\min G_i$ for every $i\in \{0,..., j\}$.
\end{enumerate}
Let $(a_n)$ be the unique sequence of reals such that
\begin{enumerate}
\item[(b)] $a_n=0$ if $n\notin G_0\cup ...\cup G_j$ and
\item[(c)] $v_{k_i}=\sum_{n\in G_i} a_n e^X_n$ for every $i\in\{0,...,j\}$.
\end{enumerate}
Notice that for every $l\in\{0,...,j\}$ and every $m\in\nn$ with $m\in\range(v_{k_l})$
we have
\begin{equation} \label{e6}
\| \sum_{n=m_l}^m a_n x_n \|_X \leq \|v_{k_l}\|_{E_X} \leq C.
\end{equation}
We select $p\in\nn$ such that
\[ \big\| \sum_{i=0}^j v_{k_i} \big\|_{E_X}=\big\| \sum_{n=0}^p
a_n x_n\big\|_X \]
and we distinguish the following cases.
\medskip

\noindent \textsc{Case 1:} \textit{$p\in \range(v_{k_0})$}.
Using (\ref{e6}), we see that
\[ \big\| \sum_{i=0}^j v_{k_i} \big\|_{E_X} =
\big\| \sum_{n=m_0}^p a_n x_n\big\|_X \leq C.\]
\medskip

\noindent \textsc{Case 2:} \textit{there exists $l\in \{1,...,j\}$ such that
$p\in\range(v_{k_l})$}. Using our hypotheses on the sequence $(v_k)$ and
inequality (\ref{e6}), we get that
\begin{eqnarray*}
\big\| \sum_{i=0}^j v_{k_i} \big\|_{E_X} & = & \big\| \sum_{i=0}^{l-1}
\sum_{n\in G_i} a_n x_n + \sum_{n=m_l}^p a_n x_n \big\|_X \\
& \leq & \sum_{i=0}^{l-1} \big\| \sum_{n\in G_i} a_n x_n \big\|_X  +
\big\| \sum_{n=m_l}^p a_n x_n \big\|_X \\
& = & \sum_{i=0}^{l-1} \|Q_X(v_{k_i})\|_X + \big\| \sum_{n=m_l}^p a_n x_n \big\|_X
\leq 2+C.
\end{eqnarray*}
\medskip

\noindent \textsc{Case 3:} \textit{for every $i\in\{0,...,j\}$ we have that
$p\notin\range(v_{k_i})$}. In this case we notice that there exists $l\in \{0,...,j\}$
such that $\range(v_{k_i})<\{p\}$ for every $i\in \{0,...,l\}$ while $\{p\}<\range(v_{k_i})$
otherwise. Using this observation we see that
\[ \big\| \sum_{i=0}^j v_{k_i} \big\|_{E_X} = \big\| \sum_{i=0}^l \sum_{n\in G_i}
a_n x_n\big\|_X \leq \sum_{i=0}^l \|Q_X(v_{k_i})\|_X \leq 2.\]
\medskip

\noindent The above cases are exhaustive, and so, the proof is completed.
\end{proof}
\begin{prop} \label{p18}
Let $(v_k)$ be a bounded block sequence of $(e^X_n)$. If $(Q_X(v_k))$ is weakly null,
then $(v_k)$ is also weakly null.
\end{prop}
For the proof of Proposition \ref{p18} we will need the following ``unconditional"
version of Mazur's Theorem.
\begin{lem} \label{l19}
Let $(v_k)$ be a weakly null sequence in a Banach space $V$. Then for
every $\ee>0$ there exist $k_0< ...< k_j$ in $\nn$ and
$\lambda_0,...,\lambda_j$ in $\rr_+$ with $\sum_{i=0}^j \lambda_i=1$
and such that
\[ \max\{ \lambda_i: 0\leq i\leq j\}\leq \ee \]
and
\[ \max\Big\{ \big\| \sum_{i\in F} \lambda_i v_{k_i} \big\|:
F\subseteq \{0,...,j\} \Big\} \leq \ee. \]
\end{lem}
\begin{proof}
Clearly we may assume that $V=\overline{\mathrm{span}}\{v_k:k\in\nn\}$,
and so, we may also assume that $V$ is a subspace of $C(2^\nn)$. Therefore,
each $v_k$ is a continuous function on $2^\nn$ and the norm of $V$ is the
usual $\|\cdot\|_\infty$ norm. By Lebesgue's dominated convergence
Theorem, a sequence $(f_k)$ in $C(2^\nn)$ is weakly null if and only if
$(f_k)$ is bounded and pointwise convergent to $0$. Hence, setting $y_k=|v_k|$
for every $k\in\nn$, we see that the sequence $(y_k)$ is weakly null.
Therefore, using Mazur's Theorem, it is possible to find $k_0<...<k_j$
in $\nn$ and $\lambda_0,...,\lambda_j$ in $\rr_+$ with $\sum_{i=0}^j \lambda_i=1$
and such that
\[ \max\{\lambda_i: 0\leq i\leq j\}\leq \ee\]
and $\|\sum_{i=0}^j \lambda_i y_{k_i}\|_\infty\leq\ee$. Noticing that
\[ \max\Big\{ \big\|\sum_{i\in F} \lambda_i v_{k_i} \big\|_\infty : F\subseteq
\{0,...,j\}\Big\} \leq \big\|\sum_{i=0}^j \lambda_i y_{k_i}\big\|_\infty\leq\ee\]
the proof is completed.
\end{proof}
We proceed to the proof of Proposition \ref{p18}.
\begin{proof}[Proof of Proposition \ref{p18}]
We will argue by contradiction. So, assume that the sequence $(Q_X(v_k))$
is weakly null while the sequence $(v_k)$ is not. We select $C\geq 1$ such
that $\|v_k\|_{E_X}\leq C$ for every $k\in\nn$. By passing to a subsequence
of $(v_k)$ if necessary, we find $e^*\in E^*_X$ and $\delta>0$ such that
$e^*(v_k)\geq \delta$ for every $k\in\nn$. This property implies that
\begin{enumerate}
\item[(a)] for every vector $z$ in $\mathrm{conv}\{v_k:k\in\nn\}$
we have $\|z\|_{E_X}\geq \delta$.
\end{enumerate}
We apply Lemma \ref{l19} to the weakly null sequence $(Q_X(v_k))$ and
$\ee=\delta\cdot (4C)^{-1}$ and we find $k_0<...<k_j$ in $\nn$ and
$\lambda_0, ..., \lambda_j$ in $\rr_+$ with $\sum_{i=0}^j \lambda_i=1$
and such that
\begin{equation} \label{e7}
\max\{ \lambda_i: 0\leq i\leq j\}\leq \frac{\delta}{4C}
\end{equation}
and
\begin{equation} \label{e8}
\max\Big\{ \big\| \sum_{i\in F} \lambda_i Q_X(v_{k_i}) \big\|_X:
F\subseteq \{0,...,j\} \Big\} \leq \frac{\delta}{4C}.
\end{equation}
Since $\|v_k\|_{E_X}\leq C$ for every $k\in\nn$, inequality (\ref{e7})
implies that
\begin{enumerate}
\item[(b)] $\|\lambda_i v_{k_i}\|_{E_X}\leq \delta/4$ for every $i\in\{ 0,..., j\}$.
\end{enumerate}
We define
\[ w=\sum_{i=0}^j \lambda_i v_{k_i}\in\mathrm{conv}\{v_k:k\in\nn\}.\]
We will show that $\|w\|_{E_X}\leq\delta/2$. This estimate contradicts
property (a) above.

To this end we will argue as in the proof of Proposition \ref{p17}. First we set
\begin{enumerate}
\item[(c)] $G_i=\supp(v_{k_i})$ and $m_i=\min G_i$ for every $i\in\{0,...,j\}$
\end{enumerate}
and we let $(a_n)$ be the unique sequence of reals such that
\begin{enumerate}
\item[(d)] $a_n=0$ if $n\notin G_0\cup ... \cup G_j$ and
\item[(e)] $\lambda_i v_{k_i}=\sum_{n\in G_i} a_n e^X_n$ for every $i\in \{0,...,j\}$.
\end{enumerate}
Using (b), we see that if $l\in\{0,...,j\}$ and $m\in\nn$ with $m\in\range(v_{k_l})$, then
\begin{equation} \label{e9}
\big\| \sum_{n=m_l}^m a_n x_n \big\|_X \leq \|\lambda_l v_{k_l}\|_{E_X} \leq \frac{\delta}{4}.
\end{equation}
We select $p\in\nn$ such that
\[ \|w\|_{E_X} = \big\| \sum_{n=0}^p a_n x_n \big\|_X \]
and, as in the proof of Proposition \ref{p17}, we consider the following three cases.
\medskip

\noindent \textsc{Case 1:} \textit{$p\in \range(v_{k_0})$}. Using (\ref{e9}),
we see that
\[ \|w\|_{E_X}= \big\| \sum_{n=m_0}^p a_n x_n\big\|_X \leq \frac{\delta}{4}. \]
\medskip

\noindent \textsc{Case 2:} \textit{there exists $l\in \{1,..., j\}$
such that $p\in\range(v_{k_l})$}. In this case the desired estimate will be
obtained combining inequalities (\ref{e8}) and (\ref{e9}). Specifically,
let $F=\{0,..., l-1\}$ and notice that
\begin{eqnarray*}
\|w\|_{E_X} & = & \big\| \sum_{i=0}^{l-1} \sum_{n\in G_i} a_n x_n +
\sum_{n=m_l}^p a_n x_n \big\|_X \\
& \leq & \big\| \sum_{i\in F} \sum_{n\in G_i} a_n x_n \big\|_X
+ \big\| \sum_{n=m_l}^p a_n x_n \big\|_X \\
& = & \big\| \sum_{i\in F} \lambda_i Q_X(v_{k_i}) \big\|_X +
\big\| \sum_{n=m_l}^p a_n x_n \big\|_X \\
& \stackrel{(\ref{e8})}{\leq} & \frac{\delta}{4C} +
\big\| \sum_{n=m_l}^p a_n x_n \big\|_X \stackrel{(\ref{e9})}{\leq}
\frac{\delta}{4C}+ \frac{\delta}{4}\leq \frac{\delta}{2}.
\end{eqnarray*}
\medskip

\noindent \textsc{Case 3:} \textit{for every $i\in\{0,...,j\}$ we have
that $p\notin\range(v_{k_i})$}. In this case we will use only inequality
(\ref{e8}). Indeed, there exists $l\in\{0,...,j\}$
such that $\range(v_{k_i})<\{p\}$ if $i\in\{0,...,l\}$ while
$\{p\}<\range(v_{k_i})$ otherwise. Setting $H=\{0,...,l\}$,
we see that
\[ \|w\|_{E_X} = \big\| \sum_{i=0}^l \sum_{n\in G_i} a_n x_n \big\|_X
= \big\| \sum_{i\in H} \lambda_i Q_X(v_{k_i}) \big\|_X
\stackrel{(\ref{e8})}{\leq} \frac{\delta}{4C}\leq \frac{\delta}{4}. \]
\medskip

\noindent The above cases are exhaustive, and so, $\|w\|_{E_X}\leq \delta/2$.
As we have already pointed out, this estimate yields a contradiction. The proof
is completed.
\end{proof}

\subsection{Proof of Theorem \ref{t3}}

Let $X$ be a separable Banach space. In what follows $(x_n)$ will be the
sequence in $X$ which is used to define the space $E_X$.
\medskip

\noindent (i) It is straightforward.
\medskip

\noindent (ii) We have already noticed that $\|Q_X||=1$. To see that $Q_X$ is
onto, observe that the image of the closed unit ball of $E_X$ under the operator
$Q_X$ contains the set $\{x_n:n\in\nn\}$ and therefore it is dense in the closed
unit ball of $X$.
\medskip

\noindent (iii) Let $Y$ be an infinite-dimensional subspace of $E_X$ and assume
that the operator $Q_X:Y\to X$ is strictly singular. Using a standard sliding hump
argument we find a block subspace $V$ of $E_X$ and a subspace $Y'$ of $Y$ with
$V$ isomorphic to $Y'$ and such that the operator $Q_X:V\to X$ is strictly
singular. Hence, we may select a normalized block sequence $(v_k)$ of $(e_n^X)$
with $v_k\in V$ and $\|Q_X(v_k)\|_X\leq 2^{-k}$ for every $k\in\nn$.
By Proposition \ref{p17}, the sequence $(v_k)$ is equivalent to the standard
unit vector basis of $c_0$ and the result follows.
\medskip

\noindent (iv) This part was essentially observed in \cite{Sch}. We reproduce
the argument for completeness. So, let $(w_k)$ be a normalized basic
sequence in $X$. The sequence $(x_n)$ is dense in the unit sphere of $X$.
Therefore it is possible to select an infinite subset $N=\{n_0< n_1<... \}$
of $\nn$ such that the subsequence $(x_{n_k})$ of $(x_n)$ determined by $N$
is basic and equivalent to $(w_k)$ (see \cite{LT}). Let $K\geq 1$
be the basis constant of $(x_{n_k})$. Let also $(e^X_{n_k})$ be the subsequence
of $(e^X_n)$ determined by $N$. Let $j\in\nn$ and $a_0,...,a_j\in\rr$ be arbitrary
and notice that
\[ \big\| \sum_{k=0}^j a_k x_{n_k} \big\|_X \leq
\big\| \sum_{k=0}^j a_k e^X_{n_k} \big\|_{E_X} =
\max_{0\leq i\leq j} \big\| \sum_{k=0}^i a_k x_{n_k} \big\|_X
\leq  K \big\| \sum_{k=0}^j a_k x_{n_k} \big\|_X. \]
Therefore, the sequence $(x_{n_k})$ is $K$-equivalent to the sequence
$(e^X_{n_k})$ and the result follows.
\medskip

\noindent (v) First we consider the relation $\sss\subseteq C(2^\nn)^\nn\times\sbs$
defined by
\[ \big((y_n),Y\big)\in\sss\Leftrightarrow (\forall n \ y_n\in Y) \text{ and }
\overline{\mathrm{span}}\{y_n:n\in\nn\}=Y.\]
The relation $\sss$ is analytic (see \cite[Lemma 2.6]{Bos2}).
Next, we apply Proposition \ref{p8} and we get a sequence
$S_n:\sbs\to C(2^\nn)$ $(n\in\nn)$ of Borel maps such that for every $X\in\sbs$
with $X\neq\{0\}$ the sequence $(S_n(X))$ is norm dense in the unit
sphere of $X$. Now notice that
\begin{eqnarray*}
(X,Y)\in\eee & \Leftrightarrow & \exists (y_n)\in C(2^\nn)^\nn \text{ with }
\big((y_n),Y\big)\in\sss \text{ and either}\\
& & \Big( X=\{0\} \text{ and } \forall k\in\nn \ \forall a_0,...,a_k\in\mathbb{Q}
\text{ we have}\\
& & \ \  \big\| \sum_{n=0}^k a_n y_n \big\|_\infty = \max_{0\leq n \leq k} |a_n|
\Big) \text{ or}\\
& & \Big( X\neq\{0\} \text{ and } \forall k\in\nn \ \forall a_0,...,a_k\in\mathbb{Q}
\text{ we have}\\
& & \ \ \big\| \sum_{n=0}^k a_n y_n \big\|_\infty = \max_{0\leq m\leq k} \big\|
\sum_{n=0}^m a_n S_n(X)\big\|_\infty \Big).
\end{eqnarray*}
The above formula implies that the set $\eee$ is analytic.
\medskip

\noindent (vi) By part (ii), the space $X$ is a quotient of $E_X$.
Therefore, if $E_X^*$ is separable, then $X^*$ is also separable.
For the converse implication we argue by contradiction. So, assume
that there exists a Banach space $X$ with separable dual such that
$E_X^*$ is non-separable. Our strategy is to show that there exists
a sequence $(w_t)_{t\in\ct}$ in $E_X$ which is topologically
equivalent to the basis of James tree (see Definition \ref{d5})
and is such that its image under the operator $Q_X$ has the
same property; that is, the sequence $(Q_X(w_t))_{t\in\ct}$ will also
be topologically equivalent to the basis of James tree. As we have already
indicated in \S 2.3, this implies that $X^*$ is non-separable and yields
a contradiction.

To this end we argue as follows. First we notice that the space $X$ does
not contain a copy of $\ell_1$. Therefore, by part (iii), the space $E_X$
does not contain a copy of $\ell_1$ either. Hence, we may apply Theorem
\ref{t7} to the space $E_X$ and we get a sequence $(e_t)_{t\in \ct}$ in
$E_X$ which is topologically equivalent to the basis of James tree.
We need to replace the sequence $(e_t)_{t\in\ct}$ with another sequence
having an additional property. Specifically, let us say that
a sequence $(v_t)_{t\in\ct}$ in $E_X$ is a \textit{tree-block} if for
every $\sg\in 2^\nn$ the sequence $(v_{\sg|n})$ is a block sequence
of $(e^X_n)$. Notice that the notion of a tree-block is hereditary
with respect to dyadic subtrees; that is, if $(v_t)_{t\in\ct}$ is a
tree-block and $D=\{d_t:t\in\ct\}$ is a dyadic subtree of $\ct$,
then the sequence $(v_{d_t})_{t\in\ct}$ is also a tree-block.
\begin{claim} \label{c20}
There exists a sequence $(v_t)_{t\in\ct}$ in $E_X$ which
is topologically equivalent to the basis of James tree and a tree-block.
\end{claim}
\begin{proof}[Proof of Claim \ref{c20}]
We select $C\geq 1$ such that $C^{-1}\leq \|e_t\|_{E_X}\leq C$ for every $t\in\ct$.
Let $s\in\ct$ be arbitrary. There exists an infinite antichain $A$ of $\ct$
such that $s\sqsubset t$ for every $t\in A$. The sequence $(e_t)_{t\in\ct}$
is topologically equivalent to the basis of James tree, and so, the
sequence $(e_t)_{t\in A}$ is weakly null. Using this observation, we may
recursively construct a dyadic subtree $R=\{r_t:t\in\ct\}$ of $\ct$ and a
tree-block sequence $(v_t)_{t\in\ct}$ in $E_X$ such that
$\|e_{r_t}-v_t\|_{E_X} \leq (2C)^{-|t|+1}$ for every $t\in\ct$.
Clearly the sequence $(v_t)_{t\in\ct}$ is as desired.
\end{proof}
\begin{claim} \label{c21}
There exist a dyadic subtree $S_0$ of $\ct$ and a constant $\Theta\geq 1$ such that
$\Theta^{-1}\leq \|Q_X(v_t)\|_X\leq \Theta$ for every $t\in S_0$.
\end{claim}
\begin{proof}[Proof of Claim \ref{c21}]
Let $K\geq 1$ be such that $\|v_t\|_{E_X}\leq K$ for every $t\in\ct$.
We will show that there exist $s_0\in\ct$ and $\theta>0$ such that for every
$t\in \ct$ with $s_0\sqsubseteq t$ we have $\|Q_X(v_t)\|_X\geq \theta$.
In such a case, set $S_0=\{s_0^{\con}t: t\in\ct\}$ and
$\Theta=\max\{\theta^{-1}, K\}$ and notice that $S_0$ and $\Theta$ satisfy
the requirements of the claim.

To find the node $s_0$ and the constant $\theta$
we will argue by contradiction. So, assume that for every $s\in\ct$ and every
$\theta>0$ there exists $t\in\ct$ with $s\sqsubseteq t$ and such that
$\|Q_X(v_t)\|_X\leq \theta$. Hence, we may select a sequence
$(t_k)$ in $\ct$ such that for every
$k\in\nn$ we have
\begin{enumerate}
\item[(a)] $t_k\sqsubset t_{k+1}$ and
\item[(b)] $\|Q_X(v_{t_k})\|_X\leq 2^{-k}$.
\end{enumerate}
By (a) above, the set $\{t_k:k\in\nn\}$ is a chain while, by Claim \ref{c20},
the sequence $(v_t)_{t\in\ct}$ is semi-normalized and a tree block. Therefore,
the sequence $(v_{t_k})$ is a semi-normalized block sequence of $(e^X_n)$.
By Proposition \ref{p17} and (b) above, we see that the sequence $(v_{t_k})$
is equivalent to the standard unit vector basis of $c_0$, and so, it is
weakly null. By Claim \ref{c20}, however, the sequence $(v_t)_{t\in\ct}$
is topologically equivalent to the basis of James tree. Since the set
$\{t_k:k\in\nn\}$ is a chain, the sequence $(v_{t_k})$ must be
non-trivial weak* Cauchy. This yields a contradiction.
\end{proof}
\begin{claim} \label{c22}
There exists a dyadic subtree $S_1$ of $\ct$ with $S_1\subseteq S_0$ and
such that for every infinite chain $\{t_0\sqsubset t_1\sqsubset ...\}$ of $S_1$
the sequence $(Q_X(v_{t_n}))$ is basic.
\end{claim}
\begin{proof}[Proof of Claim \ref{c22}]
By Claim \ref{c20} and Claim \ref{c21}, we see that for every $s\in S_0$ there
exists an infinite antichain $A$ of $S_0$ with $s\sqsubset t$ for every $t\in A$
and such that the sequence $(Q_X(v_t))_{t\in A}$ is semi-normalized and weakly null.
Using this observation and the classical procedure of Mazur for selecting
basic sequences (see \cite{LT}), the claim follows.
\end{proof}
\begin{claim} \label{c23}
There exists a dyadic subtree $S_2$ of $\ct$ with $S_2\subseteq S_1$ and
such that for every infinite chain $\{t_0\sqsubset t_1\sqsubset ...\}$ of $S_2$
the sequence $(Q_X(v_{t_n}))$ is weak* Cauchy.
\end{claim}
\begin{proof}[Proof of Claim \ref{c23}]
Let
\[ \xxx=\big\{ c\in [S_1]_{\chains}: \text{the sequence } (Q_X(v_t))_{t\in c}
\text{ is  weak* Cauchy}\big\}. \]
The set $\xxx$ is co-analytic (see \cite{Ste} for more details). Therefore, by Theorem
\ref{t4}, there exists a dyadic subtree $S_2$ of $\ct$ with $S_2\subseteq S_1$ and
such that $[S_2]_{\chains}$ is monochromatic. It is enough to show that
$[S_2]_{\chains}\cap \xxx\neq\varnothing$. Recall that the space $X$ does not
contain a copy of $\ell_1$. Therefore, by Rosenthal's Dichotomy \cite{Ro},
we may find an infinite chain $c$ of $S_2$ such that the sequence
$(Q_X(v_t))_{t\in c}$ is weak* Cauchy and the result follows.
\end{proof}
Let $S_2$ be the dyadic subtree of $\ct$ obtained by Claim \ref{c23}
and let $\{s_t:t\in\ct\}$ be the canonical representation of $S_2$. We are
in the position to define the sequence $(w_t)_{t\in\ct}$ we mentioned in
the beginning of the proof. Specifically, we set
\[ w_t= v_{s_t}\]
for every $t\in\ct$. By Claim \ref{c20} and Fact \ref{f6}, the sequence
$(w_t)_{t\in\ct}$ is topologically equivalent to the basis of James tree
and a tree block. The final claim is the following.
\begin{claim} \label{c24}
The sequence $(Q_X(w_t))_{t\in\ct}$ is topologically equivalent to the
basis of James tree.
\end{claim}
\begin{proof}[Proof of Claim \ref{c24}]
By Claim \ref{c21}, the sequence $(Q_X(w_t))_{t\in\ct}$ is semi-normalized.
Notice also that for every infinite antichain $A$ of $\ct$ the sequence
$(Q_X(w_t))_{t\in A}$ is weakly null.

Let $\sg\in 2^\nn$. By Claim \ref{c23}, the sequence $(Q_X(w_{\sg|n}))$
is weak* convergent to an element $x^{**}_\sg\in X^{**}$. First we notice
that $x^{**}_\sg\neq 0$. Indeed, the sequence $(w_{\sg|n})$ is a
semi-normalized block sequence of $(e^X_n)$ which is weak* convergent
to an element $w^{**}_\sg\in E_X^{**}\setminus E_X$. If $x^{**}_\sg=0$,
then by Proposition \ref{p18} we would have that $(w_{\sg|n})$ is weakly
null. Hence $x^{**}_\sg\neq 0$. Next we observe that $x^{**}_\sg\in
X^{**}\setminus X$. Indeed, by Claim \ref{c22}, the sequence $(Q_X(w_{\sg|n}))$
is basic. Therefore, if the sequence $(Q_X(w_{\sg|n}))$ was weakly
convergent to an element $x\in X$, then necessarily we would have that $x=0$.
This possibility, however, is ruled out by the previous reasoning, and so,
$x^{**}_\sg\in X^{**}\setminus X$.

Finally suppose, towards a contradiction, that there exist $\sg,\tau\in 2^\nn$
with $\sg\neq\tau$ and such that $x^{**}_\sg=x^{**}_\tau$. In such a case
it is possible to select two sequences $(s_n)$ and $(t_n)$ in $\ct$ such
that the following are satisfied.
\begin{enumerate}
\item[(a)] $s_n\sqsubset s_{n+1}\sqsubset \sg$ for every $n\in\nn$.
\item[(b)] $t_n\sqsubset t_{n+1}\sqsubset \tau$ for every $n\in\nn$.
\item[(c)] Setting $z_n=w_{s_n}-w_{t_n}$ for every $n\in\nn$, we have that
the sequence $(z_n)$ is a semi-normalized block sequence of $(e^X_n)$.
\end{enumerate}
Our assumption that $x^{**}_\sg=x^{**}_\tau$ reduces to the fact that the
sequence $(Q_X(z_n))$ is weakly null. By (c) above, we may apply Proposition
\ref{p18} to infer that the sequence $(z_n)$ is also weakly null.
Therefore, the sequences $(w_{\sg|n})$ and $(w_{\tau|n})$ are weak*
convergent to the same element of $E^{**}_X$. This contradicts the fact
that the sequence $(w_t)_{t\in\ct}$ is topologically equivalent to
the basis of James tree. The proof is completed.
\end{proof}
As we have already indicated, Claim \ref{c24} yields a contradiction.
This completes the proof of part (vi) of Theorem \ref{t3}, and so,
the entire proof is completed.

\subsection{Consequences}

We isolate below three corollaries of Theorem \ref{t3}. The second one
will be of particular importance in the next section.
\begin{cor} \label{c25}
Let $Z$ be a minimal Banach space not containing a copy of $c_0$. If $X$
is a separable Banach space not containing a copy of $Z$, then $E_X$
does not contain a copy of $Z$ either.
\end{cor}
\begin{proof}
Follows immediately by part (iii) of Theorem \ref{t3}.
\end{proof}
\begin{cor} \label{c26}
Let $Z$ be a minimal Banach space not containing a copy of $c_0$ and let $\aaa$ be an
analytic subset of $\ncn_Z\cap \ncn_{\ell_1}$. Then there exists a subset
$\bbb$ of $\ncn_Z\cap\ncn_{\ell_1}$ with the following properties.
\begin{enumerate}
\item[(i)] The set $\bbb$ is analytic.
\item[(ii)] Every $Y\in\bbb$ has a Schauder basis.
\item[(iii)] For every $X\in\aaa$ there exists $Y\in\bbb$ such that $X$
is a quotient of $Y$.
\end{enumerate}
\end{cor}
\begin{proof}
Let $\eee$ be the set defined in part (v) of Theorem \ref{t3}. We define
$\bbb\subseteq \sbs$ by the rule
\[ Y\in \bbb \Leftrightarrow \exists X \ [X\in \aaa \text{ and }
(X,Y)\in\eee]. \]
The set $\bbb$ is clearly analytic. Invoking parts (i) and (ii) of Theorem
\ref{t3} and Corollary \ref{c25}, we see that $\bbb$ is as desired.
\end{proof}
\begin{cor} \label{c27}
There exists a map $f:\omega_1\to\omega_1$ such that for every countable
ordinal $\xi$ and every separable Banach space $X$ with $\phi_{\ncn_{\ell_1}}(X)
\leq \xi$ the space $X$ is a quotient of a Banach space $Y$ with a Schauder
basis satisfying $\phi_{\ncn_{\ell_1}}(Y)\leq f(\xi)$.
\end{cor}
\begin{proof}
We define the map $f:\omega_1\to\omega_1$ as follows. Fix a countable ordinal
$\xi$ and consider the set
\[ \aaa_\xi=\{ X\in\sbs: \phi_{\ncn_{\ell_1}}(X)\leq \xi\}.\]
By Theorem \ref{t9}, the map $\phi_{\ncn_{\ell_1}}:\ncn_{\ell_1}\to\omega_1$
is a co-analytic rank on $\ncn_{\ell_1}$. Hence the set $\aaa_\xi$ is
analytic (in fact Borel - see \cite{Kechris}). We apply Corollary \ref{c26}
to the space $Z=\ell_1$ and the analytic set $\aaa_\xi$ and we get an
analytic subset $\bbb$ of $\ncn_{\ell_1}$ such that for every $X\in\aaa_\xi$
there exists $Y\in\bbb$ with a Schauder basis and having $X$ as quotient.
By boundedness, there exists a countable ordinal $\zeta$ such that
\[ \sup\{ \phi_{\ncn_{\ell_1}}(Y): Y\in\bbb\} =\zeta. \]
We define $f(\xi)=\zeta$. Clearly the map $f$ is as desired.
\end{proof}

\subsection{Comments}

By a well-known result due to W. J. Davis, T. Fiegel, W. B. Johnson and A.
Pe{\l}czy\'{n}ski \cite{DFJP}, if $X$ is a Banach space with separable dual,
then $X$ is a quotient of a Banach space $V_X$ with a \textit{shrinking}
Schauder basis. By Theorem \ref{t3}, the space $E_X$ has a Schauder basis,
separable dual and admits $X$ as quotient. We point out, however, that the
natural Schauder basis $(e^X_n)$ of $E_X$ is \textit{not} shrinking. On the
other hand, the subspace structure of $E_X$ is very well understood. The space
$V_X$ mentioned above is defined using the interpolation techniques developed
in \cite{DFJP} and it is not clear which are the isomorphic types of its subspaces.

We would also like to make some comments about the proof of the
separability of the dual of $E_X$. As we have already indicated, our strategy
was to construct a sequence $(w_t)_{t\in\ct}$ in $E_X$ which is topologically
equivalent to the basis of James tree and is such that its image under the
operator $Q_X$ has the same property; in other words, the operator
$Q_X$ fixes a copy of this basic object. This kind of reasoning can
be applied to a more general framework. Specifically, let $Y$ and $Z$
be separable Banach spaces and $T:Y\to Z$ be a bounded linear operator.
There are a number of problems in Functional Analysis which boil down
to understand when the dual operator $T^*$ of $T$ has non-separable
range. Using the combinatorial tools developed in \cite{ADK1} and
an analysis similar to the one in the present paper,
it can be shown that if $Y$ does not contain a copy of $\ell_1$, then
the operator $T^*$ has non-separable range if and only if $T$ fixes a
copy of a sequence which is topologically equivalent to the basis of James tree.


\section{Proof of the main result}

In this section we will give the proof of Theorem \ref{t2} stated in the
introduction. The proof will be based on the following, more detailed,
result.
\begin{thm} \label{t28}
Let $Z$ be a minimal Banach space not containing a copy of $c_0$ and
$\aaa$ be an analytic subset of $\ncn_Z\cap \ncn_{\ell_1}$. Then there
exists a Banach space $V\in \ncn_Z\cap \ncn_{\ell_1}$ with a Schauder basis
which is surjectively universal for the class $\aaa$. Moreover, if $X\in\aaa$
has the bounded approximation property, then $X$ is isomorphic to a
complemented subspace of $V$.
\end{thm}
Let us point out that the assumption on the complexity of the set $\aaa$
in Theorem \ref{t28} is optimal. Notice also that if $E$ is any Banach space
with a Schauder basis, then the set of all $X\in\aaa$ which are isomorphic to
a complemented subspace of $E$ is contained in the set of all $X\in\aaa$ having
the bounded approximation property. Therefore, the ``moreover" part of the above
result is optimal too.
\begin{proof}[Proof of Theorem \ref{t28}]
Since $Z$ is minimal, there exists $1<p<+\infty$ such that $Z$ does not contain
a copy of $\ell_p$. We fix such a $p$. We apply Lemma \ref{l13} to the
space $Z$ and the analytic set $\aaa$ and we get a subset $\ddd$ of
$\ncn_Z\cap \ncn_{\ell_1}$ such that the following are satisfied.
\begin{enumerate}
\item[(a)] The set $\ddd$ is analytic.
\item[(b)] Every $Y\in\ddd$ has a Schauder basis.
\item[(c)] For every $X\in\aaa$ with the bounded approximation property there
exists $Y\in\ddd$ such that $X$ is isomorphic to a complemented subspace of $Y$.
\end{enumerate}
Next we apply Corollary \ref{c26} to the space $Z$ and the analytic set $\aaa$
and we get a subset $\bbb$ of $\ncn_Z\cap \ncn_{\ell_1}$ with the following
properties.
\begin{enumerate}
\item[(d)] The set $\bbb$ is analytic.
\item[(e)] Every $Y\in\bbb$ has a Schauder basis.
\item[(f)] For every $X\in\aaa$ there exists $Y\in\bbb$ such that $X$
is a quotient of $Y$.
\end{enumerate}
We set $\ccc=\bbb\cup\ddd$ and we notice that $\ccc\subseteq\ncn_Z\cap \ncn_{\ell_1}$.
By (a) and (d), the set $\ccc$ is analytic while, by (b) and (e), every $Y\in\ccc$
has a Schauder basis. The desired space $V$ is the $p$-amalgamation space of the
class $\ccc$ obtained by Theorem \ref{t14}. It rests to check that $V$ has the
desired properties. Notice, first, that $V$ has a Schauder basis.
\begin{claim} \label{c29}
The space $V$ is surjectively universal for the class $\aaa$.
\end{claim}
\begin{proof}[Proof of Claim \ref{c29}]
Let $X\in \aaa$ arbitrary. By (f), there exists a space $Y\in\bbb$ such that
$X$ is a quotient of $Y$. We fix a quotient map $Q:Y\to X$. Next we observe
that the space $V$ contains a complemented copy of $Y$. Therefore, it is
possible to find a subspace $E$ of $V$, a projection $P:V\to E$
and an isomorphism $T:E\to Y$. Let $Q':V\to X$ be the operator defined
by $Q'=Q\circ T \circ P$ and notice that $Q'$ is onto. Hence, $X$
is a quotient of $V$ and the result follows.
\end{proof}
\begin{claim} \label{c30}
We have $V\in \ncn_Z\cap \ncn_{\ell_1}$.
\end{claim}
\begin{proof}[Proof of Claim \ref{c30}]
We will show that $V$ does not contain a copy of $Z$ (the proof of the
fact that $V$ does not contain a copy of $\ell_1$ is identical). We will
argue by contradiction. So, assume that there exists a subspace $W$
of $V$ which is isomorphic to $Z$. By the choice of $p$, we see that
$W$ does not contain a copy of $\ell_p$. Therefore, by Theorem \ref{t14},
there exist $Y_0,..., Y_n$ in the class $\ccc$ such that $W$ is isomorphic
to a subspace of $Y_0\oplus ...\oplus Y_n$. There exist an infinite-dimensional
subspace $W'$ of $W$ and $i_0\in\{0,...,n\}$ such that $W'$ is isomorphic
to a subspace of $Y_{i_0}$. Since $Z$ is minimal, we get that
$Y_{i_0}$ must contain a copy of $Z$. This contradicts the fact that
$\ccc\subseteq \ncn_Z$, and so, the claim is proved.
\end{proof}
Finally, we notice that if $X\in\aaa$ has the bounded approximation property,
then, by (c) above and Theorem \ref{t14}, the space $X$ is isomorphic to a
complemented subspace of $V$. This shows that the space $V$ has the
desired properties. The proof of Theorem \ref{t28} is completed.
\end{proof}
We proceed to the proof of Theorem \ref{t2}.
\begin{proof}[Proof of Theorem \ref{t2}]
Let $\ccc\subseteq\sbs$.
\medskip

\noindent (i)$\Rightarrow$(ii) Assume that there exists a separable Banach
space $Y$ not containing a copy of $\ell_1$ which is surjectively universal
for the class $\ccc$. The space $Y$ does not contain a copy of $\ell_1$, and
so, $\phi_{\ncn_{\ell_1}}(Y)<\omega_1$. Moreover, every space in the class
$\ccc$ is a quotient of $Y$. Therefore, by Lemma \ref{l10}, we get that
\[ \sup\{\phi_{\ncn_{\ell_1}}(X):X\in\ccc\} \leq \phi_{\ncn_{\ell_1}}(Y) <\omega_1.\]

\noindent (ii)$\Rightarrow$(iii) Let $\xi$ be a countable ordinal such
that $\sup\{\phi_{\ncn_{\ell_1}}(X):X\in\ccc\}=\xi$. By Theorem \ref{t9},
the map $\phi_{\ncn_{\ell_1}}:\ncn_{\ell_1}\to\omega_1$ is a co-analytic
rank on the set $\ncn_{\ell_1}$. It follows that the set
\[ \aaa=\{V\in\sbs: \phi_{\ncn_{\ell_1}}(V)\leq \xi\} \]
is a Borel subset of $\ncn_{\ell_1}$ (see \cite{Kechris}) and clearly
$\ccc\subseteq \aaa$.
\medskip

\noindent (iii)$\Rightarrow$(i) Assume that there exists an analytic subset
$\aaa$ of $\ncn_{\ell_1}$ with $\ccc\subseteq \aaa$. We apply Theorem
\ref{t28} for $Z=\ell_1$ and the class $\aaa$ and we get a Banach space
$V$ with a Schauder basis which does not contain a copy of $\ell_1$ and is
surjectively universal for the class $\aaa$. A fortiori, the space $V$
is surjectively universal for the class $\ccc$ and the result follows.
\end{proof}


\section{A related result and open problems}

Let us recall the following notion (see \cite[Definition 90]{AD}).
\begin{defn} \label{d31}
A class $\ccc\subseteq\sbs$ is said to be \emph{strongly bounded}
if for every analytic subset $\aaa$ of $\ccc$ there exists $Y\in\ccc$
which is universal for the class $\aaa$.
\end{defn}
This is a quite strong structural property. It turned out, however, that
many natural classes of separable Banach spaces are strongly bounded.

Part of the research in this paper grew out from our attempt
to find natural instances of the ``dual" phenomenon. The ``dual"
phenomenon is described in abstract form in the following definition.
\begin{defn} \label{d32}
A class $\ccc\subseteq\sbs$ is said to be \emph{surjectively strongly bounded}
if for every analytic subset $\aaa$ of $\ccc$ there exists $Y\in\ccc$ which is
surjectively universal for the class $\aaa$.
\end{defn}
So, according to Definition \ref{d32}, Theorem \ref{t28} has the following consequence.
\begin{cor} \label{c33}
Let $Z$ be a minimal Banach space not containing a copy of $c_0$. Then
the class $\ncn_{\ell_1}\cap \ncn_{Z}$ is surjectively strongly bounded.
\end{cor}
The following proposition provides two more natural examples.
\begin{prop} \label{p34}
The class $\mathrm{REFL}$ of separable reflexive Banach spaces and the class
$\mathrm{SD}$ of Banach spaces with separable dual are surjectively strongly
bounded.
\end{prop}
Proposition \ref{p34} follows combining a number of results already existing
in the literature, and so instead of giving a formal proof we will only give
a guideline. To see that the class $\mathrm{REFL}$ is surjectively strongly
bounded, let $\aaa$ be an analytic subset of $\mathrm{REFL}$ and consider
the \textit{dual class} $\aaa^*$ of $\aaa$ defined by
\[ Y\in\aaa^* \Leftrightarrow \exists X\in\aaa \text{ with } Y
\text{ isomorphic to } X^*. \]
The set $\aaa^*$ is analytic (see \cite{D1}) and
$\aaa^*\subseteq\mathrm{REFL}$. Since the class $\mathrm{REFL}$ is
strongly bounded (see \cite{DF}), there exists a separable reflexive Banach
space $Z$ which is universal for the class $\aaa^*$. Therefore,
every space $X$ in $\aaa$ is a quotient of $Z^*$. The referee suggested
that, alternatively, one can use the universality results obtained in
\cite{OSZ}.

The argument for the class $\mathrm{SD}$ is somewhat different and uses
the parameterized version of the Davis-Fiegel-Johnson-Pe{\l}czy\'{n}ski
construction due to Bossard, as well as, an idea already employed in the
proof of Theorem \ref{t28}. Specifically, let $\aaa$ be an analytic
subset of $\mathrm{SD}$. By the results in \cite{DFJP} and \cite{Bos1},
there exists an analytic subset $\bbb$ of Banach spaces with a shrinking
Schauder basis such that for every $X\in\aaa$ there exists $Y\in\bbb$
having $X$ as quotient. It is then possible to apply the machinery
developed in \cite{AD} to obtain a Banach space $E$ with a shrinking Schauder
basis that contains a complemented copy of every space in the class $\bbb$.
By the choice of $\bbb$, we see that the space $E$ is surjectively universal
for the class $\aaa$.

Although, by Theorem \ref{t2}, we know that the class $\ncn_{\ell_1}$ is
surjectively strongly bounded, we should point out that it is not
known whether the class $\ncn_{\ell_1}$ is strongly bounded. We close
this section by mentioning the following related problems.
\begin{problem} \label{pr1}
Is it true that every separable Banach space $X$ not containing a copy
of $\ell_1$ embeds into a space $Y$ with a Schauder basis and not
containing a copy of $\ell_1$?
\end{problem}
\begin{problem} \label{pr2}
Does there exist a map $g:\omega_1\to\omega_1$ such that for every
countable ordinal $\xi$ and every separable Banach space $X$ with
$\phi_{\ncn_{\ell_1}}(X)\leq \xi$ the space $X$ embeds into a Banach space
$Y$ with a Schauder basis satisfying $\phi_{\ncn_{\ell_1}}(Y)\leq g(\xi)$?
\end{problem}
\begin{problem} \label{pr3}
Is the class $\ncn_{\ell_1}$ strongly bounded?
\end{problem}
We notice that an affirmative answer to Problem \ref{pr2} can be used to
provide an affirmative answer to Problem \ref{pr3} (to see this combine
Theorem \ref{t9} and Theorem \ref{t14} stated in \S 2).

It seems reasonable to conjecture that the above problems have an affirmative
answer. Our optimism is based on the following facts. Firstly, Problem \ref{pr3}
is known to be true within the category of Banach spaces with a Schauder
basis (see \cite{AD}). Secondly, it is known that for every minimal Banach
space $Z$ not containing a copy $\ell_1$ the class $\ncn_Z$ is strongly
bounded (see \cite{D2}).
\medskip

\noindent \textbf{Acknowledgments.} This work was done during my visit
at Texas A$\&$M University the academic year 2008-2009. I would like
to thank the Department of Mathematics for the financial support and
the members of the Functional Analysis group for providing excellent
working conditions.


\end{document}